\documentclass[11pt]{article}
\usepackage[margin=1in]{geometry}
\usepackage{amsmath,amsthm,amssymb,amsfonts,mathtools}
\usepackage{bm}
\usepackage{natbib}
\usepackage{hyperref}
\usepackage{booktabs}
\hypersetup{colorlinks=true,linkcolor=blue,citecolor=blue,urlcolor=blue}
\title{Bayesimax Theory:\\ Selecting Priors by Minimizing Total Information}
\author{Sitaram Vangala\thanks{Department of Medicine Statistics Core, David Geffen School of Medicine, University of California, Los Angeles, Los Angeles, CA. E-mail: \texttt{svangala@mednet.ucla.edu}.}}
\date{September 7, 2025}

\theoremstyle{plain}
\newtheorem{theorem}{Theorem}[section]
\newtheorem{lemma}[theorem]{Lemma}
\newtheorem{proposition}[theorem]{Proposition}
\newtheorem{corollary}[theorem]{Corollary}
\theoremstyle{definition}
\newtheorem{definition}[theorem]{Definition}
\newtheorem{assumption}[theorem]{Assumption}
\theoremstyle{remark}
\newtheorem{remark}[theorem]{Remark}

\newcommand{\E}{\mathbb{E}}

\newcommand{\R}{\mathbb{R}}
\newcommand{\KL}{\mathrm{KL}}

\begin{document}
\maketitle

\begin{abstract}
We introduce \emph{Bayesimax theory}, a paradigm for objective Bayesian analysis which selects priors by applying minimax theory to \emph{prior disclosure} games. In these games, the uniquely optimal strategy for a Bayesian agent upon observing the data is to reveal their prior. As such, the prior chosen by minimax theory is, in effect, the implicit prior of minimax agents. Since minimax analysis disregards prior information, this prior is arguably noninformative. We refer to minimax solutions of certain prior disclosure games as \emph{Bayesimax priors}, and we classify a statistical procedure as Bayesimax if it is a Bayes rule with respect to a Bayesimax prior.\par
Under regular conditions, if a decision rule is minimax, then it is a Bayes rule under priors which maximize the minimum Bayes risk. We study games leveraging strictly proper scoring rules to induce posterior (and thereby prior) revelation. In such games, the minimum Bayes risk equals the conditional (generalized) entropy of the parameter given the data. Bayesimax theory therefore prescribes conditional entropy maximization. As conditional entropy equals marginal entropy (prior uninformativeness) minus mutual information (data informativeness), Bayesimax priors effectively minimize \emph{total} information.\par
We provide a rigorous formulation of these ideas, characterize sufficient conditions for regularity and identifiability, and investigate asymptotics and conjugate family examples. We next describe a generic Monte Carlo algorithm for estimating conditional entropy under a given prior. Finally, we compare and contrast Bayesimax theory with various related proposals from the objective Bayes and robust Bayes literature.
\end{abstract}

\section{Introduction}\label{sec:intro}

\paragraph{Context and motivation.}
Objective Bayesian analysis seeks priors minimizing subjective influence while preserving coherent \emph{a posteriori} decision-making. Canonical proposals include \emph{Jeffreys priors} for parameterization invariance \citep{jeffreys1946}, \emph{maximum entropy priors} for least-committal specification under constraints \citep{jaynes1957}, \emph{probability matching priors} for frequentist calibration of posterior quantiles \citep{welch1963, datta2005}, and \emph{reference priors} that maximize the information contributed by the data \citep{bernardo1979, bergerbernardo1992, kasswasserman1996, bergersun2009, consonni2018}. In a parallel line of research, statistical decision theory formalizes rational post-data choice via Bayes and minimax decision rules, with foundational treatments in \citet{wald1950}, \citet{ferguson1967}, and \citet{berger1985}. A classical bridge between these traditions utilizes \emph{least favorable priors}, under which Bayes rules are minimax with respect to a given loss function. 

Existing objective Bayes devices typically optimize either a prior-centric criterion (e.g., invariance or entropy) or a data-centric criterion (e.g., mutual information) in pursuit of noninformativeness, rather than reverse-engineering optimal noninformative post-data behavior. This paper proposes such a device, and characterizes the class of priors it rationalizes.

\paragraph{Prior disclosure games and Bayesimaxity.}
We introduce \emph{Bayesimax theory}, which selects priors by applying minimax analysis to \emph{prior disclosure} games. A prior disclosure game is a post-data decision problem in which, upon observing the data, the uniquely optimal strategy for a Bayesian is to report their prior. Accordingly, the minimax solution of such a game may be interpreted as the implicit prior of minimax theory. This prior is arguably noninformative in the sense that it reverse-engineers a decision principle that disregards prior information. We term minimax solutions of certain prior disclosure games \emph{Bayesimax priors}, and we say that a statistical procedure is \emph{Bayesimax} if it is a Bayes rule with respect to a Bayesimax prior.

In this paper, we study a natural and tractable subclass of prior disclosure games that leverage \emph{strictly proper scoring rules}. Such scoring rules incentivize truthful probabilistic statements. We consider decision problems in which the choice set is a space of priors, and the loss function applies a strictly proper scoring rule to the induced posterior. In such problems, a Bayesian minimizes posterior expected loss by selecting their own prior. Under a standard identifiability condition, this truth-telling solution is unique. See \citet{gneiting2007} for a general treatment of strictly proper scoring rules and \citet{dawid2012, dawid2015} for posterior/local scoring and decision-theoretic uses.

\paragraph{Risk maximization and entropy.}
Let $\Theta$ denote the parameter with prior $\pi$, $X$ data from a dominated experiment $\{P_\theta:\theta\in\Omega\}$, $m^\pi$ the prior-predictive distribution of $X$, and $\Pi_x^\pi$ the posterior under $\pi$ given $X = x$. In a disclosure game employing a strictly proper score $S$ on distributions over $\Omega$, the loss at $\Theta = \theta$ for decision rule $\delta$ is $-S(\Pi_x^{\delta(x)}, \theta)$. The \emph{Bayes risk} of $\delta$ under $\pi$ is the posterior expected loss for $\delta(X)$, averaged with respect to $m^\pi$. In such games, the \emph{minimum Bayes risk} (over decision rules) under $\pi$ equals the \emph{conditional (generalized) entropy},
\[
\E_{X \sim m^\pi}\!\big[\,H_S(\Pi_X^\pi)\,\big],
\]
where $H_S$ is the generalized entropy associated with $S$ \citep{gneiting2007}. Under regular conditions, if a decision rule is minimax, then it is a Bayes rule under priors which maximize the minimum Bayes risk. Hence, in games featuring a strictly proper scoring rule, minimax solutions maximize conditional entropy. For the logarithmic score in particular, $H_S$ is Shannon entropy ($H$), which recalls the following decomposition \citep{coverthomas2006}:
\[
\underset{\text{(conditional entropy)}}{H(\Theta\mid X)} \;=\; \underset{\text{(marginal entropy)}}{H(\Theta)}\;-\;\underset{\text{(mutual information)}}{I(\Theta;X).}
\]
Hence, a Bayesimax prior minimizes \emph{total} information by maximizing prior uninformativeness, $H(\Theta)$, and minimizing data informativeness, $I(\Theta;X)$. Intuitively, minimax theory selects a prior that, on average, results in a posterior as diffuse as the model tolerates.

\paragraph{Positioning within objective Bayes and robust Bayes literature.}
The Bayesimax criterion recalls, but differs qualitatively from, existing proposals. Maximum entropy priors maximize \emph{prior} uncertainty \citep{jaynes1957}, while ignoring informativeness of the data. Reference priors \emph{maximize}, rather than minimize, data informativeness \citep{bernardo1979, bergerbernardo1992, bergersun2009}, while ignoring prior uncertainty. In contrast, Bayesimax priors minimize information deriving from both the prior and the data. The justification for this criterion is conceptually adjacent to robust Bayes and $\Gamma$-minimax ideas \citep{berger1985, berger1994robust, insua2000}, yet distinct: rather than minimizing worst-case risk in statistical problems, Bayesimax priors do so only in prior disclosure games. A statistical procedure is Bayesimax so long as it is fully Bayesian under a Bayesimax prior, regardless of whether the procedure itself is minimax.

\paragraph{Contributions and paper organization.}
Section~\ref{sec:prelim} introduces notation and provides background on strictly proper scoring rules, including the generalized entropy $H_S$ and its associated divergence. Section~\ref{sec:game} formalizes prior disclosure games and demonstrates our main results. Section~\ref{sec:existence} outlines sufficient conditions for regularity and identifiability. Section~\ref{sec:asymp} studies asymptotics for the logarithmic score in regular finite-dimensional models. Section~\ref{sec:examples} illustrates Bayesimax analysis in conjugate families. Section~\ref{sec:practice} describes a generic Monte Carlo algorithm for estimating conditional entropy under a given prior. Section~\ref{sec:discussion} concludes by revisiting competing proposals from the objective Bayes and robust Bayes literature.

\section{Preliminaries: notation and background}\label{sec:prelim}
Let $(\Omega,\mathcal{T})$ be a standard Borel parameter space and let $\mathcal{P}=\{P_\theta:\theta\in\Omega\}$ be a dominated statistical experiment on measurable space $(\mathcal{X},\mathcal{A})$ with $\sigma$-finite dominating measure $\lambda$. For $\pi$ a Borel probability measure on $(\Omega,\mathcal{T})$, write the prior predictive (marginal) as
$\mathrm{d}m^\pi/\mathrm{d}\lambda = \int_\Omega p_\theta(x)\,\mathrm{d}\pi$,
where $p_\theta=\mathrm{d}P_\theta/\mathrm{d}\lambda$.
Assume Bayes posteriors $\Pi_x^\pi$ exist for $m^\pi$-almost all $x$ and are measurable in $x$. Let $\Lambda$ be a space of probability measures on $(\Omega,\mathcal{T})$.

\begin{definition}[Strictly proper scoring rules and generalized entropy]\label{def:score}
A scoring rule $S$ assigns to each probability measure $Q$ on $(\Omega,\mathcal{T})$ a measurable function $S(Q,\cdot):\Omega\to\R\cup\{-\infty\}$. It is \emph{proper} if $\E_{Y\sim P}[S(P,Y)]\ge \E_{Y\sim P}[S(Q,Y)]$ for all $P,Q$; \emph{strictly proper} if equality implies $Q=P$. The associated \emph{generalized entropy} is $H_S(P)=-\E_{Y\sim P}[S(P,Y)]$. The \emph{score divergence} is $D_S(P\Vert Q)=\E_{Y\sim P}[S(P,Y)-S(Q,Y)]\ge 0$, with equality iff $P=Q$.
\end{definition}

\begin{remark}
For the logarithmic score $S(Q,y)=\log q(y)$ (density w.r.t.\ Lebesgue measure), $H_S$ is differential Shannon entropy, $H$, and $D_S$ is Kullback--Leibler divergence, $\KL$. For quadratic/Brier scores one recovers $L^2$ entropies and divergences; for spherical/pseudospherical scores, corresponding norms induce the entropy \citep{gneiting2007}.
\end{remark}

\section{Prior disclosure games and Bayesimax priors}\label{sec:game}
Fix a strictly proper score $S$ on probability measures over $(\Omega,\mathcal{T})$.
Define a \emph{prior disclosure game} as follows.
A \emph{decision rule} is a measurable selector $\delta:\mathcal{X}\to\Lambda$; upon observing $X=x$, the player selects a probability measure $\delta(x)\in\Lambda$.
The \emph{loss} at $(\theta,x,a)$ for $a \in \Lambda$ is
\begin{equation}\label{eq:loss}
\ell_S(\theta,x,a) \;=\; -\, S\big(\Pi_x^{a},\,\theta\big).
\end{equation}
The \emph{frequentist risk} of $\delta$ when $\Theta = \theta$ is
\begin{equation}\label{eq:freqrisk}
\mathcal{R}_S(\delta, \theta) \;=\; \E_{X\sim P_\theta}\big[\ell_S(\theta,X,\delta(X))\big].
\end{equation}
For a prior $\pi\in\Lambda$, the \emph{Bayes risk} of $\delta$ is
\begin{equation}
R_S(\delta,\pi) \;=\; \E_{\Theta\sim\pi}\big[\mathcal{R}_S(\delta, \Theta)\big]
\;=\; -\,\E_{X\sim m^\pi}\Big[ \E_{\Theta\sim\Pi_X^\pi} \Big[S\big(\Pi_X^{\delta(X)},\Theta\big)\Big]\Big].
\end{equation}
Define the \emph{minimum Bayes risk} of prior $\pi$ as $r_S(\pi)=\inf_{\delta} R_S(\delta,\pi)$. $\delta$ is a \emph{Bayes rule} under $\pi$ if $R_S(\delta,\pi) = r_S(\pi)$.
$\delta^M$ is \emph{minimax} if
\begin{equation}
\sup_{\theta \in \Omega}\mathcal{R}_S(\delta^M, \theta) \;=\; \inf_\delta \sup_{\theta \in \Omega}\mathcal{R}_S(\delta, \theta).
\end{equation}
A prior $\pi$ is \emph{least favorable} if it admits a Bayes rule $\delta_\pi$ such that for all $\theta \in \Omega$, $\mathcal{R}_S(\delta_\pi,\theta) \leq r_S(\pi)$. 

\begin{assumption}[Prior identifiability]\label{ass:ident}
For $m^\pi$-a.e. $x$, the map $a \mapsto \Pi^a_x$ is injective on $\Lambda$.
\end{assumption}

\begin{proposition}[Truth-telling for Bayesians]\label{thm:truthtelling}
Under assumption~\ref{ass:ident}, for any prior $\pi\in\Lambda$ and any strictly proper score $S$, Bayes rules for the prior disclosure game select $\pi$ almost surely:
\begin{equation}
\delta^\star(x)\equiv \pi\qquad (m^\pi\text{-a.s. }x),
\end{equation}
and the minimum Bayes risk equals
\begin{equation}\label{eq:bayesvalue}
r_S(\pi) = \E_{X\sim m^\pi}\Big[ H_S\big(\Pi_X^\pi\big)\Big].
\end{equation}
\end{proposition}

\begin{proof}
For each $x$, strict propriety of $S$ implies that
$\E_{\Theta\sim \Pi_x^\pi} \Big[ S\big(\Pi_x^\pi,\Theta\big) \Big] \ge \E_{\Theta\sim \Pi_x^\pi} \Big[ S\big(Q,\Theta\big) \Big]$
with equality iff $Q=\Pi_x^\pi$. Taking $Q=\Pi_X^{\delta(X)}$ and integrating over $X \sim m^\pi$ yields that a pointwise Bayes risk minimizer is $\delta(x)=\pi$, and the minimum is \eqref{eq:bayesvalue}. Identifiability ensures that any rule achieving the minimum $m^\pi$-a.s.\ must coincide with $\delta^\star$.
\end{proof}

\begin{proposition}[Conditional (generalized) entropy decomposition]\label{prop:decomp}
For any $\pi\in\Lambda$, $r_S(\pi)$ admits the decomposition
\begin{equation}
r_S(\pi)=H_S(\pi) - I_S(\pi),\qquad
I_S(\pi)=\E_{X\sim m^\pi}\Big[ D_S\big(\Pi_X^\pi\Vert \pi\big)\Big]\ge 0.
\end{equation}
\end{proposition}

\begin{proof}
By definition \ref{def:score}, for each $x$,
$H_S(\Pi_x^\pi)=-\E_{\Theta\sim\Pi_x^\pi} \Big[ S(\Pi_x^\pi,\Theta) \Big]
= -\E_{\Theta\sim\Pi_x^\pi} \Big[ S(\pi,\Theta) \Big] - D_S(\Pi_x^\pi\Vert\pi)$.
Integrating over $X\sim m^\pi$ and applying the tower property gives
$r_S(\pi) = H_S(\pi) - I_S(\pi)$.
\end{proof}

\begin{definition}[Bayesimax priors and procedures]\label{def:bayesimax}
A \emph{Bayesimax prior} is any maximizer
\begin{equation}
\pi^\star \in \arg\max_{\pi\in\Lambda} r_S(\pi) = \arg\max_{\pi\in\Lambda} \E_{X\sim m^\pi}\big[H_S(\Pi_X^\pi)\big],
\end{equation}
for some strictly proper score $S$. A \emph{Bayesimax procedure} for a statistical decision problem is a Bayes rule under $\pi^\star$ for that problem. If the maximum is not realized, a sequence $(\pi_i)$ with $r_S(\pi_i)\uparrow\sup_{\pi \in \Lambda} r_S(\pi)$ is a \emph{nearly Bayesimax sequence}.
\end{definition}

\begin{lemma}[Bayes-minimax]\label{lm:lfminimax}
Let $\pi$ be a least favorable prior. Then $\delta_\pi$ is minimax, and every minimax rule is a Bayes rule under $\pi$.
\end{lemma}

\begin{proof}
Suppose $\delta_\pi$ is not minimax. Then there is a $\delta$ such that $\sup_{\theta \in \Omega}\mathcal{R}_S(\delta, \theta) < \sup_{\theta \in \Omega}\mathcal{R}_S(\delta_\pi, \theta)$. Since $\pi$ is least favorable, and $R_S(\delta, \pi) \leq \sup_{\theta \in \Omega}\mathcal{R}_S(\delta, \theta)$, we have $R_S(\delta, \pi) <r_S(\pi)$ (contradiction). Hence, $\delta_\pi$ is minimax. Next, let $\delta$ be minimax. Then $\sup_{\theta \in \Omega}\mathcal{R}_S(\delta, \theta) = \sup_{\theta \in \Omega}\mathcal{R}_S(\delta_\pi, \theta)$. Since $\pi$ is least favorable, and $R_S(\delta, \pi) \leq \sup_{\theta \in \Omega}\mathcal{R}_S(\delta, \theta)$, it follows that $R_S(\delta, \pi) \leq r_S(\pi)$. Hence, every minimax rule is a Bayes rule under $\pi$.
\end{proof}

\begin{assumption}[Least favorable regularity]\label{ass:lfreg}
There exists a $\pi_m \in \Lambda$ which is least favorable with respect to the prior disclosure game.
\end{assumption}

\begin{proposition}[Minimax characterization in the disclosure game]\label{thm:minimax}
Under assumption~\ref{ass:lfreg}, all and only Bayesimax priors are least favorable for the disclosure game. Consequently, every minimax rule in the disclosure game is a Bayes rule under every Bayesimax prior.
\end{proposition}

\begin{proof}
For any $\pi \in \Lambda$, $r_S(\pi) \leq R_S(\delta_{\pi_m}, \pi) \leq \sup_{\theta \in \Omega}\mathcal{R}_S(\delta_{\pi_m}, \theta)$. Suppose $\pi_m$ is not Bayesimax. Then there is a $\pi \in \Lambda$ such that $r_S(\pi_m) < r_S(\pi)$, in which case we have $r_S(\pi_m) < \sup_{\theta \in \Omega}\mathcal{R}_S(\delta_{\pi_m}, \theta)$ (contradiction). It follows that every least favorable prior is Bayesimax. In particular, if $\pi$ is Bayesimax, then $r_S(\pi_m) = r_S(\pi)$. Since $\pi_m$ is least favorable, $r_S(\pi) \leq R_S(\delta_{\pi_m}, \pi) \leq \sup_{\theta \in \Omega}\mathcal{R}_S(\delta_{\pi_m}, \theta) \leq r_S(\pi)$. It follows that every Bayesimax prior is least favorable. Accordingly, lemma~\ref{lm:lfminimax} entails that every minimax rule in the disclosure game is a Bayes rule under every Bayesimax prior. 
\end{proof}

\begin{corollary}[Minimax solution]\label{cor:minimax}
Under assumptions~\ref{ass:ident} and \ref{ass:lfreg}, if a decision rule $\delta$ is minimax in the disclosure game, then for all Bayesimax priors $\pi$, $\delta(x) = \pi\text{ }(m^\pi\text{-a.s. }x)$.
\end{corollary}

\begin{proof}
This follows immediately from propositions~\ref{thm:truthtelling} and \ref{thm:minimax}.
\end{proof}

\section{Identifiability and regularity}\label{sec:existence}

We now present sufficient conditions for assumptions~\ref{ass:ident} and \ref{ass:lfreg}.

\subsection*{Sufficient conditions for prior identifiability}

\begin{lemma}[Common domination with negligible zeros]\label{thm:dom_ident}
Suppose there is a $\sigma$-finite measure $\nu$ where, for all $a \in \Lambda$, $a \ll \nu$. If for $m^\pi$-a.e.\ $x$, $p_\theta(x)>0$ for $\nu$-a.e.\ $\theta$, then assumption~\ref{ass:ident} holds.
\end{lemma}

\begin{proof}
Consider $a_1, a_2 \in \Lambda$ and $x$ in the $m^\pi$-full set, where $\Pi_x^{a_1}=\Pi_x^{a_2}$. Since $a_1, a_2 \ll \nu$, it follows that $\Pi_x^{a_1}, \Pi_x^{a_2} \ll \nu$. Any density of $\Pi_x^{a_1}$ (with respect to $\nu$) is also a density of $\Pi_x^{a_2}$, and is therefore proportional to both $p_\theta(x)\frac{\mathrm{d}a_1}{\mathrm{d}\nu}(\theta)$ and $p_\theta(x)\frac{\mathrm{d}a_2}{\mathrm{d}\nu}(\theta)$ for $\nu$-a.e.\ $\theta$. Since $p_\theta(x)>0$ $\nu$-a.e., it follows that $\frac{\mathrm{d}a_1}{\mathrm{d}\nu}(\theta) \propto \frac{\mathrm{d}a_2}{\mathrm{d}\nu}(\theta)$ $\nu$-a.e., meaning $a_1 = a_2$.
\end{proof}

\subsection*{Sufficient conditions for least favorable regularity}

We first note a pair of jointly sufficient conditions for assumption~\ref{ass:lfreg}.

\begin{lemma}[Supremum identity and saddle point existence]\label{prop:lfp_suff}
Suppose (i) $\sup_{\theta\in\Omega}\,\mathcal{R}_S(\delta,\theta) \;=\; \sup_{\pi\in\Lambda} R_S(\delta,\pi)$ for all decision rules $\delta$, and (ii) there exists a saddle point $(\delta^\star,\pi^\star)$ with $R_S(\delta^\star,\pi)\le R_S(\delta^\star,\pi^\star)\le R_S(\delta,\pi^\star)$ for all decision rules $\delta$ and $\pi \in \Lambda$. Then assumption~\ref{ass:lfreg} holds.
\end{lemma}

\begin{proof}
Since $R_S(\delta^\star,\pi^\star)\le R_S(\delta,\pi^\star)$ for all $\delta$, $\delta^\star$ is a Bayes rule under $\pi^\star$. Combining (i) and (ii), we have $\underset{\theta \in \Omega}{\sup}\mathcal{R}_S(\delta^\star,\theta)\;=\;\underset{\pi \in \Lambda}{\sup}R_S(\delta^\star,\pi)\;=\;R_S(\delta^\star,\pi^\star)\;=\;r_S(\pi^\star)$. Hence, $\pi^\star$ is least favorable.
\end{proof}

\begin{remark}
A convenient way to guarantee (i) is to require that $\Lambda$ includes all point masses $\{\delta_\theta:\theta\in\Omega\}$, or includes them in its weak closure together with upper semicontinuity of $\pi\mapsto R_S(\delta,\pi)$. In that case, we have $\sup_{\theta\in\Omega}\mathcal{R}_S(\delta,\theta)\;=\; \sup_{\theta\in\Omega} R_S(\delta,\delta_\theta)\;=\;\sup_{\pi\in\Lambda} R_S(\delta,\pi)$.
\end{remark}

\noindent We now present sufficient conditions for (ii). Throughout, $\Lambda$ is endowed with the weak topology.

\begin{lemma}[Saddle points via compactness and semicontinuity]\label{prop:sion}
Suppose
\begin{enumerate}
\item[(a)] $\Lambda$ is convex and weakly compact (e.g., $\Omega$ compact, or $\Lambda$ tight and weakly closed).
\item[(b)] For each decision rule $\delta$, $\pi\mapsto R_S(\delta,\pi)$ is upper semicontinuous and affine on $\Lambda$.
\item[(c)] For each $\pi\in\Lambda$, the map $\delta\mapsto R_S(\delta,\pi)$ is lower semicontinuous and convex on the convex set of (randomized) decision rules.
\item[(d)] For each $\pi\in\Lambda$, $\inf_\delta R_S(\delta,\pi)$ is realized (e.g., by the truth-telling rule; proposition~\ref{thm:truthtelling}).
\end{enumerate}
Then there exists a saddle point $(\delta^\star,\pi^\star)$.
\end{lemma}

\begin{proof}[Proof sketch]
Given (a)–(c), Sion's minimax theorem \citep{sion1958} yields $\inf_\delta\sup_{\pi \in \Lambda} R_S(\delta,\pi)\;=\;\sup_{\pi \in \Lambda}\inf_\delta R_S(\delta,\pi)$, and since $\pi\mapsto R_S(\delta,\pi)$ is upper semicontinuous on compact $\Lambda$, the supremum over $\pi$ is realized. Given (d), the infimum over $\delta$ is attained pointwise in $\pi$. Standard saddle-point arguments \citep{wald1950, ferguson1967, berger1985} then give $(\delta^\star,\pi^\star)$.
\end{proof}

\begin{remark}[How to check (a)–(d) in this setting]
A dominated model and strictly proper scoring rule suffice for measurability and finiteness of risks. Condition (b) holds because $R_S(\delta,\pi)$ is affine in $\pi$ and upper semicontinuous under weak convergence when $\theta\mapsto \mathcal{R}_S(\delta,\theta)$ is bounded above and upper semicontinuous (by the portmanteau theorem or a dominated-convergence envelope). Condition (c) is the usual convexity/l.s.c.\ of risk in the (randomized) decision rule; for strictly proper scores (including the logarithmic score), $-S$ is convex in its distribution argument, and convexity propagates through mixing of actions at each $x$. Condition (d) follows from proposition~\ref{thm:truthtelling}, which identifies a measurable minimizer for each $\pi$.
\end{remark}

\subsection*{Nearly Bayesimax sequences}

In cases where a Bayesimax prior does not exist, a nearly Bayesimax sequence $(\pi_i)$ usually does. Accordingly, $\delta$ is \emph{nearly Bayesimax} if $R_S(\delta, \pi_i) = r_S(\pi_i)$ for some large $i$.

\begin{lemma}[Sequence existence]\label{thm:nbseq}
If $\Lambda$ is nonempty and $\sup_{\pi\in\Lambda} r_S(\pi)$ is finite, then there is a nearly Bayesimax sequence in $\Lambda$.
\end{lemma}

\begin{proof}
For each $i \in \mathbb{N}$, choose $\pi_i \in \Lambda$ so that $r_S(\pi_i) > \sup_{\pi\in\Lambda} r_S(\pi) - \frac{1}{i}$. Then $(\pi_i)$ is nearly Bayesimax. Since $\Lambda$ is nonempty and $\sup_{\pi\in\Lambda} r_S(\pi) < \infty$, such a choice must exist for all $i \in \mathbb{N}$.
\end{proof}

\section{Asymptotics for the logarithmic score}\label{sec:asymp}
We study the logarithmic score with i.i.d.\ data $X^{(n)}=(X_1,\dots,X_n)$ from a regular $d$-dimensional parametric model with Fisher information $\mathcal{I}(\cdot)$ and prior $\pi$. Let $\mathcal{I}_\pi(\cdot)$ be the negative Hessian of the log-density of $\pi$ w.r.t.\ Lebesgue measure.
Under local asymptotic normality~\citep{lecamyang2000}, when $\Theta = \theta$ the posterior is asymptotically $\mathcal{N}( \cdot , \Sigma_{\theta, n}^\pi)$, where $\Sigma_{\theta, n}^\pi = (n\mathcal{I}(\theta)+\mathcal{I}_\pi(\theta))^{-1}$.

\begin{proposition}[Asymptotic decomposition]\label{thm:asymp}
Given local asymptotic normality, for large $n$
\begin{equation}\label{eq:asymp}
\E_{X^{(n)} \sim m^\pi}\Big[H(\Pi^{\pi}_{X^{(n)}})\Big] \;\approx\; \frac{d}{2}\log(2 \pi e) - \frac{1}{2}\E_{\Theta \sim \pi}\Big[\log(\det[n\mathcal{I}(\Theta)+\mathcal{I}_\pi(\Theta)])\Big].
\end{equation}
\end{proposition}

\begin{proof}[Proof]
Obtain normal approximation via $2^\text{nd}$ order expansion of log-likelihood and log-prior about $\theta$, with $\Theta = \theta + \frac{h}{\sqrt{n}}$. Compute the entropy of $\mathcal{N}( \cdot , \Sigma_{\Theta, n}^\pi)$. No terms depend on $X^{(n)}$, hence integration over $X^{(n)} \sim m^\pi$ is just integration over $\Theta \sim \pi$ by the tower property, giving \eqref{eq:asymp}.
\end{proof}

\begin{remark}[The shape of Bayesimax priors]\label{rem:shape}
Only the second term in \eqref{eq:asymp} depends on $\pi$. Hence, Bayesimax solutions balance two objectives---concentrating on minimal Fisher information regions, and minimizing prior curvature---but their relative importance depends on sample size---$\mathcal{O}(n)$ v. $\mathcal{O}(1)$. Accordingly, as $n$ increases Bayesimax priors transition from reasonably flat to sharply concentrated, with probability mass transferring from high-information to low-information regions.
\end{remark}

\begin{remark}[Contra Jeffreys]\label{rem:jeffreys}
For \emph{very} large $n$, maximizing conditional entropy is equivalent to maximizing the cross-entropy of $\pi$ relative to the Jeffreys prior. The Bayesimax criterion is therefore diametrically opposed to the Jeffreys criterion in the asymptotic limit.
\end{remark}

\section{Conjugate prior examples}\label{sec:examples}
We illustrate the logarithmic score case in conjugate families; throughout, $n$ denotes sample size.

\subsection{Normal--Normal with known variance: $X^{(n)} \mid \Theta \sim \mathcal{N}(\Theta,\sigma^2)$ and $\Theta\sim\mathcal{N}(\mu,\tau^2)$}
The posterior is $\Theta\mid X^{(n)} \sim \mathcal{N}\!\big(\tilde\mu,\tilde\tau^2\big)$, with $\tilde\tau^2=(\tau^{-2}+n\sigma^{-2})^{-1}$ independent of $X^{(n)}$. Hence, the conditional entropy is
\begin{equation}
r_S(\pi)
= \tfrac12\log\big(2\pi e\,\tilde\tau^2\big).
\end{equation}
$r_S(\pi)$ is strictly increasing in $\tau^2$, with $\sup_{\pi} r_S(\pi) = \tfrac12\log\!\left(\frac{2\pi e\,\sigma^2}{n}\right)$ as $\tau^2\to\infty$. There is thus no proper maximizer; a nearly Bayesimax sequence has $\tau_i^2\to\infty$ (limiting flat prior). In this example, Fisher information ($n / \sigma^2$) does not vary with $\Theta$, meaning conditional entropy maximization asymptotically reduces to prior curvature minimization. As such, the (limiting) Bayesimax prior matches the maximum entropy, Jeffreys, and reference prior in this case.

\subsection{Beta--Bernoulli: $X^{(n)} \mid\Theta\sim\mathrm{Bernoulli}(\Theta)$ and $\Theta\sim\mathrm{Beta}(\alpha,\beta)$}\label{subsec:bb}
The posterior is $\Theta\mid X^{(n)} \sim \mathrm{Beta}(\tilde\alpha,\tilde\beta)$, where $\tilde\alpha =\alpha + \sum{X^{(n)}}$, and $\tilde\beta =\beta + n-\sum{X^{(n)}}$. Hence, the conditional entropy is
\begin{equation}
r_S(\pi)=\E_{X^{(n)} \sim m^\pi}\big[\log(B(\tilde\alpha,\tilde\beta)) - (\tilde\alpha-1)\psi(\tilde\alpha) - (\tilde\beta-1)\psi(\tilde\beta) + (\tilde\alpha+\tilde\beta-2)\psi(\tilde\alpha+\tilde\beta)\big].
\end{equation}
Differentiating inside the expectation yields gradients in terms of digamma/trigamma functions, enabling numerical optimization. Figure~\ref{fig:fig1} displays solutions for $n=1,...,200$. Consistent with remark~\ref{rem:shape}, when $n=1$ the Bayesimax prior closely approximates $U(0, 1)$, and as $n$ increases, the prior concentrates about $\Theta = \frac{1}{2}$ (since $\mathcal{I}(\Theta) = [\Theta(1 - \Theta)]^{-1}$) at an estimated rate of $\mathcal{O}(n^{-1/4})$.
\begin{figure}
  \includegraphics[width=\linewidth]{"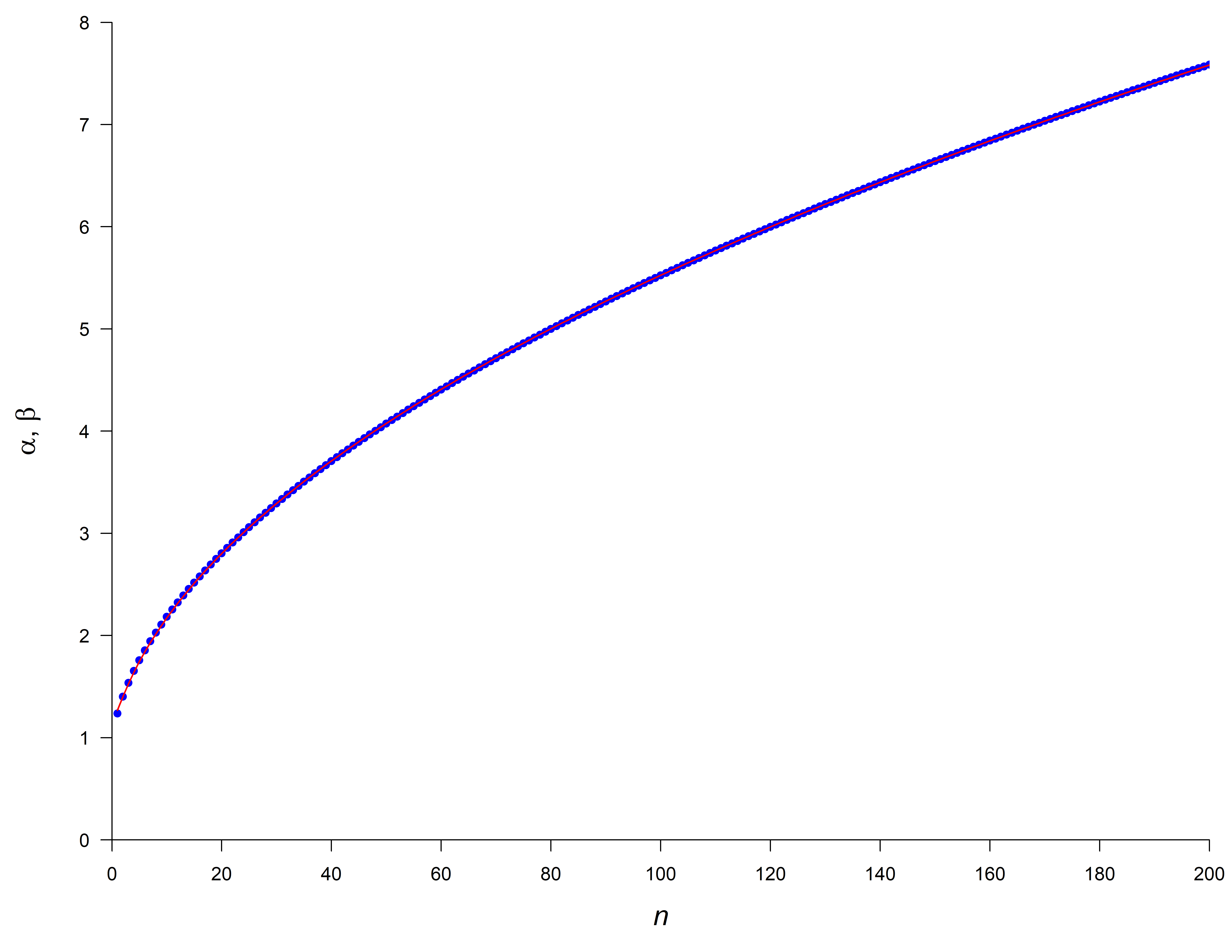"}
  \caption{Bayesimax priors for the Beta---Bernoulli model. For all $n$, maximization requires $\alpha = \beta$. Points (blue) display numerical solutions for $n=1,...,200$. Curve (red) displays a least-squares approximation, $\alpha = \beta = 0.5\sqrt{n} + 0.5 + \frac{0.3}{\sqrt{n}}$ (adj. $R^2 \approx 1$).}
  \label{fig:fig1}
\end{figure}

To evaluate the performance of the Bayesimax prior, we performed a simulation study. For each $\Theta \in \{0.01,...,0.99\}$, 100,000 datasets with $n = 100$ were sampled from the Bernoulli model. We estimated the mean squared error (MSE) of the maximum likelihood estimator, together with the posterior means under $\mathrm{Beta}(0.5, 0.5)$ (the Jeffreys/reference prior), $\mathrm{Beta}(1, 1)$ (the maximum entropy prior), and $\mathrm{Beta}(5.53, 5.53)$ (the Bayesimax prior). Figure~\ref{fig:fig2} displays the results.
\begin{figure}
  \includegraphics[width=\linewidth]{"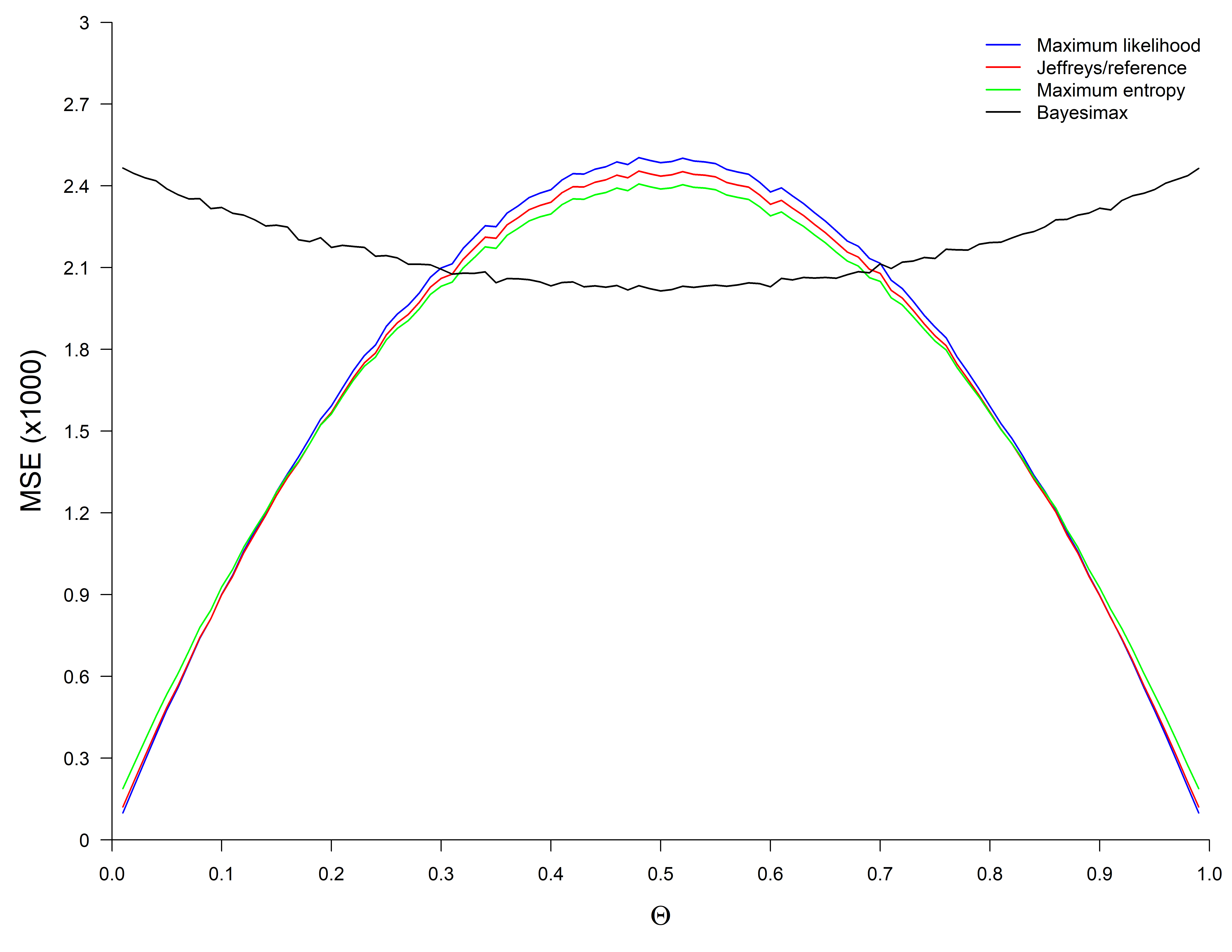"}
  \caption{Bernoulli model simulation study with $n = 100$. Each curve displays the estimated MSE of the relevant point estimator as a function of $\Theta$. For Jeffreys/reference, maximum entropy, and Bayesimax, the estimator is the corresponding posterior mean.}
  \label{fig:fig2}
\end{figure}

Unlike the other estimators, where performance rapidly improves as $\Theta$ approaches its extremes, the Bayesimax posterior mean performs best in a neighborhood of $\Theta = 0.5$, with slowly worsening MSE as $\Theta$ approaches $0$ and $1$. For large $n$, the least-squares estimate (figure~\ref{fig:fig1}) is roughly $\alpha = \beta \approx \sqrt{n/4}$. Notably, the MSE of the posterior mean under $\mathrm{Beta}(\sqrt{n/4},\sqrt{n/4})$ is constant~\citep{wasserman2004}, suggesting that the Bayesimax point estimator approximates asymptotic minimaxity with respect to squared-error loss.

\newpage\section{Estimating conditional entropy}\label{sec:practice}

We outline a simple, black-box Monte Carlo (MC) estimator of
\[
r_S(\pi)=\E_{X\sim m^\pi}\big[H_S(\Pi_X^\pi)\big],
\]
focusing on the logarithmic score where $H_S$ is (differential) Shannon entropy $H$. The method consists of an \emph{outer} prior predictive sampling loop, and an \emph{inner} posterior sampling and entropy estimation loop, utilizing a $k$-nearest-neighbor ($k$NN) entropy estimator.

\paragraph{Algorithm (Nested MC with $k$NN entropy).}
\begin{enumerate}\itemsep0.35ex
\item \textbf{Inputs:} prior $\pi$; likelihood $\{P_\theta\}$; posterior sampler $x\mapsto \text{Draw}(\Pi_x^\pi)$; integers $I$ (outer loop replications), $J$ (inner loop posterior draws), $k$ ($k$NN order).
\item \textbf{For $i=1,\ldots,I$:}
  \begin{enumerate}\itemsep0.25ex
  \item Sample $\theta_i\sim \pi$ and $x_i\sim P_{\theta_i}$.
  \item Draw $\theta_i^{(1)},\ldots,\theta_i^{(J)}\sim \Pi_{x_i}^\pi$ (e.g., MCMC with effective sample size $\mathrm{ESS}_i$).
  \item Compute $k$NN estimate $\widehat H(\Pi_{x_i}^\pi)$ from $\theta_i^{(\cdot)}$ (e.g., \cite{berrett2019}).
  \end{enumerate}
\item \textbf{Estimate:}
\[
\widehat r_S(\pi)\;=\;\frac{1}{I}\sum_{i=1}^{I}\widehat H(\Pi_{x_i}^\pi).
\]
\item \textbf{Standard error:}
\[
\mathrm{SE}[\widehat r_S(\pi)]\;=\;\Big[\frac{1}{I(I - 1)}\sum_{i=1}^{I}(\widehat H(\Pi_{x_i}^\pi) - \widehat r_S(\pi))^2\Big]^{\frac{1}{2}}.
\]
\end{enumerate}

An alternative estimator of $r_S(\pi)$ exploits proposition~\ref{prop:decomp} by separately estimating marginal entropy and mutual information from the samples in step 2(a). The primary advantage of the proposed algorithm is its simplicity and generality, but performance characteristics and optimal choice of $(I, J, k)$ necessarily depend on the model.

\subsection*{Example: intraclass correlation coefficient (ICC)}

Consider a matched pairs design where, for $s = 1, ..., n$ and $t = 1, 2$, $x_{st} = \mu + u_s + \varepsilon_{st}$, with \linebreak $u_s \mid \sigma_u^2 \sim \mathcal{N}(0, \sigma_u^2)$, and $\varepsilon_{st} \mid \sigma_\varepsilon^2 \sim \mathcal{N}(0, \sigma_\varepsilon^2)$. In terms of priors, let $\mu \sim \mathcal{N}(0, 1)$, $\sigma_u^2 \sim \mathrm{InvGamma}(1, \beta_u)$, and $\sigma_\varepsilon^2 \sim \mathrm{InvGamma}(1, \beta_\varepsilon)$. The parameter of interest is the intraclass correlation coefficient, $\rho = \frac{\sigma_u^2}{\sigma_u^2 + \sigma_\varepsilon^2}$. The conditional entropy of $\rho$ given $X$ was estimated via the proposed algorithm for low, medium, and high values of $n$, $\beta_u$, and $\beta_\varepsilon$, fixing $I = 100$, $J = 4000$, and $k = 8$. Posterior sampling was performed using Stan (\url{https://mc-stan.org/}), via the \emph{cmdstanr} package in R v. 4.4.1 (\url{https://www.r-project.org/}). Table~\ref{tbl:mcmc} displays the results. In these scenarios, our estimate of $r_S(\pi)$ peaks at either low $\beta_u$ with medium $\beta_\varepsilon$, or at medium $\beta_u$ with high $\beta_\varepsilon$, suggesting that Bayesimaxity favors priors where $\beta_\varepsilon$ is moderately larger than $\beta_u$.

\begin{table}
  \centering
  \caption{Conditional entropy estimates for the ICC in a matched pairs design.}
  \label{tbl:mcmc}
  \begin{tabular}{rrrcc}
    \toprule
    $n$ & $\beta_u$ & $\beta_\varepsilon$ & $\widehat r_S(\pi)$ & $\mathrm{SE}[\widehat r_S(\pi)]$ \\
    \midrule
      10 & 0.5 & 0.5 & -0.84 & 0.10 \\
      10 & 0.5 & 1 & -0.76 & 0.07 \\
      10 & 0.5 & 2 & -0.93 & 0.08 \\
      10 & 1 & 0.5 & -0.97 & 0.11 \\
      10 & 1 & 1 & -0.84 & 0.08 \\
      10 & 1 & 2 & -0.77 & 0.07 \\
      10 & 2 & 0.5 & -1.20 & 0.09 \\
      10 & 2 & 1 & -1.02 & 0.08 \\
      10 & 2 & 2 & -0.77 & 0.07 \\
      50 & 0.5 & 0.5 & -1.58 & 0.08 \\
      50 & 0.5 & 1 & -1.27 & 0.06 \\
      50 & 0.5 & 2 & -1.58 & 0.11 \\
      50 & 1 & 0.5 & -1.56 & 0.08 \\
      50 & 1 & 1 & -1.34 & 0.07 \\
      50 & 1 & 2 & -1.21 & 0.06 \\
      50 & 2 & 0.5 & -1.88 & 0.12 \\
      50 & 2 & 1 & -1.60 & 0.10 \\
      50 & 2 & 2 & -1.43 & 0.09 \\
      100 & 0.5 & 0.5 & -1.58 & 0.08 \\
      100 & 0.5 & 1 & -1.55 & 0.07 \\
      100 & 0.5 & 2 & -1.72 & 0.07 \\
      100 & 1 & 0.5 & -2.08 & 0.10 \\
      100 & 1 & 1 & -1.66 & 0.09 \\
      100 & 1 & 2 & -1.64 & 0.09 \\
      100 & 2 & 0.5 & -2.43 & 0.12 \\
      100 & 2 & 1 & -1.94 & 0.10 \\
      100 & 2 & 2 & -1.75 & 0.08 \\
    \bottomrule
  \end{tabular}
\end{table}

\section{Discussion}\label{sec:discussion}

\paragraph{Summary of contributions.}
We introduced prior disclosure games (\S\ref{sec:game}), studying those leveraging strictly proper scoring rules \citep{gneiting2007,dawid2012,dawid2015}, and performed minimax analysis to motivate Bayesimax priors, which maximize the minimum Bayes risk $r_S(\pi)$ in such games. We further showed that Bayesimax priors maximize conditional (generalized) entropy of the parameter given the data. We clarified conditions under which priors are identifiable and least favorable priors exist (\S\ref{sec:existence}); when these fail, nearly Bayesimax sequences provide a workable approximation (lemma~\ref{thm:nbseq}). In the case of i.i.d.\ data from regular parametric models, we obtained a simple asymptotic approximation for conditional Shannon entropy (proposition~\ref{thm:asymp}), clarifying how the Bayesimax objective contrasts with classical devices in its treatment of Fisher information (remarks~\ref{rem:shape} and \ref{rem:jeffreys}). We illustrated the criterion’s behavior in elementary models with conjugate family examples (\S\ref{sec:examples}), and described a black-box Monte Carlo estimator for $r_S(\pi)$ compatible with commonly used Markov chain Monte Carlo algorithms (\S\ref{sec:practice}). Our proposal recalls, but deviates in important respects, from several others in the literature.

\paragraph{Jeffreys priors.}
Jeffreys’ rule \citep{jeffreys1946} selects a prior density proportional to $ \sqrt{\det \mathcal{I}(\Theta)}$ for the sake of parameterization invariance. In very large samples under logarithmic scoring, maximizing $r_S(\pi)$ amounts to maximizing the cross-entropy of $\pi$ relative to the Jeffreys prior (remark~\ref{rem:jeffreys}). The Bayesimax criterion is thus diametrically opposed to Jeffreys' rule in the asymptotic limit. Mechanically, Bayesimax priors downweight regions where the Fisher information is high, while Jeffreys priors upweight these regions. At a conceptual level, Jeffreys' rule is concerned with a prior-focused geometric criterion, while Bayesimax theory is concerned not only with how informative the prior is, but also how informative the prior renders the data.

\paragraph{Maximum entropy priors.}
Jaynes’ maximum entropy program \citep{jaynes1957} chooses $\pi$ so as to maximize $H(\pi)$ subject to moment constraints. Under the log score, the Bayesimax objective is the conditional, rather than the marginal, Shannon entropy. Proposition~\ref{prop:decomp} reveals that Bayesimax theory trades off prior informativeness against data informativeness, meaning that it differs from maximizing prior entropy to the extent that this increases mutual information. Thus, while the Bayesimax criterion is partly concerned with prior entropy, its competing concern with data informativeness renders it qualitatively different from Jaynes' criterion.

\paragraph{Reference priors.}
Reference priors \citep{bernardo1979,bergerbernardo1992,kasswasserman1996,bergersun2009,consonni2018} maximize the expected information gain $\E_{X \sim m^\pi}[\KL(\Pi_X^\pi\|\pi)]$. In the case of a logarithmic score, Bayesimax priors \emph{minimize} this quantity (for fixed prior entropy). This inversion does not ensure disagreement, since Bayesimax priors balance this objective against maximizing prior entropy, but given the close relationship between Jeffreys priors and reference priors (e.g., \cite{clarkebarron1994}), disagreement is likely the norm rather than the exception (remark~\ref{rem:jeffreys}).

\paragraph{Robust Bayes.}
The robust Bayes literature investigates least favorable priors for the purpose of identifying statistical procedures which enjoy both Bayesian and minimax optimality \citep{berger1985,berger1994robust,insua2000}. Our construction yields least favorable priors specifically for prior disclosure games, rather than for the statistical decision problem of primary interest. Hence, while our approach borrows a powerful technique from the robust Bayes tradition, our objective is identifying general-purpose noninformative priors, as opposed to a diverse array of problem-specific minimaxity-conferring priors.

\paragraph{Limitations.}
Several caveats deserve emphasis. (i) \emph{Choice of score:} Different strictly proper scoring rules may identify different classes of Bayesimax priors. The log score is a natural default, but it may be that a prior which remains Bayesimax under varying scoring rules is preferable to one that is Bayesimax only under the log score. (ii) \emph{Reparameterization:} Conditional differential entropy lacks reparameterization invariance, so Bayesimax solutions may vary across parameterizations. (iii) \emph{Existence:} In many realistic problems, a Bayesimax prior may not exist; nearly Bayesimax sequences will often, but not always, perform comparably. (iv) \emph{Computation.} In high dimensions, $k$NN entropy estimators may underperform, and MCMC mixing errors may propagate. While the proposed estimator in \S\ref{sec:practice} is plug-and-play, its performance in specific applications must be carefully evaluated. (v) \emph{Design.} Bayesimax priors may vary with sample size or other design parameters, as in the Beta--Bernoulli example (\ref{subsec:bb}), in ways that competing proposals do not. Addressing these complexities is a promising direction for future research on Bayesimax theory.

\section*{Acknowledgments}
The idea of applying minimax theory to prior disclosure games, and the general solution strategy described above, were fully developed by the author. Detailed mechanics of the solution, proofs of key results, and solution characterizations were developed in a series of exchanges with OpenAI’s \emph{GPT-5 Thinking} large language model \citep{openai:gpt5thinking:2025}. A complete first draft, informed by these exchanges, was generated by the model. The present version represents a comprehensive revision by the author for accuracy, clarity, depth, and style. Had these been the contributions of a peer, co-authorship would be warranted, but the author takes full responsibility for the material presented. We would like to thank Leon Bergen and Martin Flores for helpful comments on previous drafts.

\bibliographystyle{plainnat}

\begin{thebibliography}{99}

\bibitem[Berger(1985)]{berger1985}
Berger, J. O. (1985).
\textit{Statistical Decision Theory and Bayesian Analysis} (2nd ed.).
New York, NY: Springer.

\bibitem[Berger(1994)]{berger1994robust}
Berger, J. O. (1994).
An overview of robust Bayesian analysis.
\textit{Test}, \textbf{3}(1), 5--124.

\bibitem[Berger and Bernardo(1992)]{bergerbernardo1992}
Berger, J. O., \& Bernardo, J. M. (1992).
On the development of reference priors.
In J. M. Bernardo, J. O. Berger, A. P. Dawid, \& A. F. M. Smith (Eds.),
\textit{Bayesian Statistics 4} (pp.~35--60).
Oxford: Oxford University Press.

\bibitem[Berger et~al.(2009)]{bergersun2009}
Berger, J. O., Bernardo, J. M., \& Sun, D. (2009).
The formal definition of reference priors.
\textit{Annals of Statistics}, \textbf{37}(2), 905--938.

\bibitem[Bernardo(1979)]{bernardo1979}
Bernardo, J. M. (1979).
Reference posterior distributions for Bayesian inference.
\textit{Journal of the Royal Statistical Society: Series B (Methodological)}, \textbf{41}(2), 113--147.

\bibitem[Berrett et~al.(2019)]{berrett2019}
Berrett, T. B., Samworth R. J., \& Yuan M. (2019).
Efficient multivariate entropy estimation via \emph{k}-nearest neighbour distances.
\textit{Annals of Statistics}, \textbf{47}(1), 288--318.

\bibitem[Clarke and Barron(1994)]{clarkebarron1994}
Clarke, B. S., \& Barron, A. R. (1994).
Jeffreys' prior is asymptotically least favorable under entropy risk.
\textit{Journal of Statistical Planning and Inference}, \textbf{41}(1), 37--60.

\bibitem[Consonni et~al.(2018)]{consonni2018}
Consonni, G., Fouskakis, D., Liseo, B., \& Ntzoufras, I. (2018).
Prior distributions for objective Bayesian analysis.
\textit{Bayesian Analysis}, \textbf{13}(2), 627--679.

\bibitem[Cover and Thomas(2006)]{coverthomas2006}
Cover, T. M., \& Thomas, J. A. (2006).
\textit{Elements of Information Theory} (2nd ed.).
Hoboken, NJ: Wiley-Interscience.

\bibitem[Datta and Sweeting(2005)]{datta2005}
Datta G. S., \& Sweeting T. J. (2005).
Probability matching priors.
In D. K. Dey \& C. R. Rao (Eds.),
\textit{Handbook of Statistics 25} (pp.~91-114).
Amsterdam: Elsevier.

\bibitem[Dawid et~al.(2012)]{dawid2012}
Dawid, A. P., Lauritzen, S., \& Parry, M. (2012).
Proper local scoring rules on discrete sample spaces.
\textit{Annals of Statistics}, \textbf{40}(1), 593--608.

\bibitem[Dawid and Musio(2015)]{dawid2015}
Dawid, A. P., \& Musio, M. (2015).
Bayesian model selection based on proper scoring rules.
\textit{Bayesian Analysis}, \textbf{10}(2), 479--499.

\bibitem[Ferguson(1967)]{ferguson1967}
Ferguson, T. S. (1967).
\textit{Mathematical Statistics: A Decision Theoretic Approach}.
New York, NY: Academic Press.

\bibitem[Gneiting and Raftery(2007)]{gneiting2007}
Gneiting, T., \& Raftery, A. E. (2007).
Strictly proper scoring rules, prediction, and estimation.
\textit{Journal of the American Statistical Association}, \textbf{102}(477), 359--378.

\bibitem[Insua and Ruggeri(2000)]{insua2000}
Insua, D. R., \& Ruggeri, F. (Eds.). (2000).
\textit{Robust Bayesian Analysis} (Lecture Notes in Statistics, Vol.~152).
New York, NY: Springer.

\bibitem[Jaynes(1957)]{jaynes1957}
Jaynes, E. T. (1957).
Information theory and statistical mechanics.
\textit{Physical Review}, \textbf{106}(4), 620--630.

\bibitem[Jeffreys(1946)]{jeffreys1946}
Jeffreys, H. (1946).
An invariant form for the prior probability in estimation problems.
\textit{Proceedings of the Royal Society of London. Series A}, \textbf{186}(1007), 453--461.

\bibitem[Kass and Wasserman(1996)]{kasswasserman1996}
Kass, R. E., \& Wasserman, L. (1996).
The selection of prior distributions by formal rules.
\textit{Journal of the American Statistical Association}, \textbf{91}(435), 1343--1370.

\bibitem[Le Cam and Yang(2000)]{lecamyang2000}
Le Cam, L., \& Yang, G. L. (2000).
\textit{Asymptotics in Statistics: Some Basic Concepts} (2nd ed.).
New York, NY: Springer.

\bibitem[OpenAI(2025)]{openai:gpt5thinking:2025}
OpenAI.
\newblock \emph{GPT-5 Thinking} [Large language model; via ChatGPT Plus].
\newblock Available at \url{https://chat.openai.com/}. Accessed Aug, Sep 2025.

\bibitem[Sion(1958)]{sion1958}
Sion, M. (1958).
On general minimax theorems.
\textit{Pacific Journal of Mathematics}, \textbf{8}(1), 171-176.

\bibitem[Wald(1950)]{wald1950}
Wald, A. (1950).
\textit{Statistical Decision Functions}.
New York, NY: Wiley.

\bibitem[Wasserman(2004)]{wasserman2004}
Wasserman, L. (2004).
\textit{All of Statistics: A Concise Course in Statistical Inference}.
New York, NY: Springer.

\bibitem[Welch and Peers(1963)]{welch1963}
Welch, B. L., \& Peers, H. W. (1963).
On formulae for confidence points based on integrals of weighted likelihoods.
\textit{Journal of the Royal Statistical Society: Series B (Methodological)}, \textbf{25}(2), 318--329.

\end{thebibliography}

\end{document}